\documentclass[12pt,twoside,reqno]{amsart}
\usepackage[colorlinks=true,citecolor=blue]{hyperref}
\usepackage{mathptmx, amsmath, amssymb, amsfonts, amsthm, mathptmx, enumerate, color}
\usepackage{graphicx}
\usepackage{epstopdf}

\newtheorem{theorem}{Theorem}[section]
\newtheorem{fact}{Fact}[section]
\newtheorem{corollary}{Corollary}[section]
\newtheorem{lemma}{Lemma}[section]
\newtheorem{proposition}{Proposition}[section]
\theoremstyle{definition}
\newtheorem{definition}{Definition}[section]
\newtheorem{example}{Example}[section]
\newtheorem{remark}{Remark}[section]
\newtheorem{question}{Question}[section]

\numberwithin{equation}{section}

\usepackage{empheq}
\definecolor{myblue}{rgb}{.9, .9,.9}
\newcommand*\mybluebox[1]{
	\colorbox{myblue}{\hspace{1em}#1\hspace{1em}}}

\usepackage{hyperref}
\hypersetup{
	colorlinks=true,
	linkcolor=black, % blue
	citecolor=black, % blue
	filecolor=magenta,
	urlcolor=black % cyan
}

\usepackage[capitalize, nameinlink, noabbrev]{cleveref}

\usepackage{enumitem}

\usepackage[capitalize, nameinlink, noabbrev]{cleveref}
\crefname{equation}{}{}
\crefname{chapter}{Chapter}{Chapters}
\crefname{item}{item}{items}
\crefname{figure}{Figure}{Figures}
\crefname{theorem}{Theorem}{Theorems}
\crefname{lemma}{Lemma}{Lemmas}
\crefname{proposition}{Proposition}{Propositions}
\crefname{corollary}{Corollary}{Corollarys}
\crefname{definition}{Definition}{Definitions}
\crefname{fact}{Fact}{Facts}
\crefname{example}{Example}{Examples}
\crefname{algorithm}{Algorithm}{Algorithms}
\crefname{remark}{Remark}{Remarks}
\crefname{note}{Note}{Notes}
\crefname{notation}{Notation}{Notations}
\crefname{case}{Case}{Cases}
\crefname{exercise}{Exercise}{Exercises}
\crefname{question}{Question}{Questions}
\crefname{claim}{Claim}{Claims}
\crefname{enumi}{}{}

\newcommand{\cone}{\ensuremath{\operatorname{cone}}}

\newcommand{\Pro}{\ensuremath{\operatorname{P}}}

\newcommand{\scal}[2]{\left\langle{#1},{#2}  \right\rangle}

\providecommand{\abs}[1]{\lvert#1\rvert}
\providecommand{\norm}[1]{\lVert#1\rVert}

\providecommand{\innp}[1]{\langle#1\rangle}

\providecommand{\biginnp}[1]{\big\langle#1\big\rangle}

\begin{document}
\setcounter{page}{1}

\vspace*{1.0cm}
\title[running title]
{On  angles between convex cones}
\author[H.~H.\ Bauschke, H.~Ouyang, and X.~Wang]{Heinz H.\ Bauschke$^{1,*}$, Hui\ Ouyang$^1$, Xianfu\ Wang$^{1}$}
\maketitle
\vspace*{-0.6cm}

\begin{center}
{\footnotesize {\it
$^1$Department of Mathematics, University of British Columbia,  Kelowna, B.C.\ V1V~1V7, Canada.
%\\
%$^2$Department of Mathematics, xxxxxxx University, xxxxxxxxxxxxx, Country/Region
}}\end{center}

\vskip 4mm {\small\noindent {\bf Abstract.}
There are two basic angles associated with a pair of linear subspaces:
the Diximier angle and the Friedrichs angle.
The Dixmier angle of the pair of orthogonal complements is the same as
the Dixmier angle of the original pair provided that the original pair
gives rise to a direct (not necessarily orthogonal) sum of the underlying Hilbert space.
The Friedrichs angles of the original pair
and the pair of the orthogonal complements always coincide.
These two results are due to Krein, Krasnoselskii, and Milman and
to Solmon, respectively. In 1995, Deutsch provided a very nice survey
with complete proofs and interesting historical comments.
One key result in Deutsch's survey was an inequality for Dixmier angles provided by Hundal.

In this paper, we present extensions of these results to the case when the linear subspaces 
are only required to be convex cones. It turns out that Hundal's result
has a nice conical extension while the situation is more technical
for the results by Krein et al.\ and by Solmon. Our analysis is based
on Deutsch's survey and our recent work on angles between convex sets.
Throughout, we also provide examples illustrating the sharpness of our results.

}

\renewcommand{\thefootnote}{}
\footnotetext{ $^*$Corresponding author.
\par
E-mail addresses: heinz.bauschke@ubc.ca (H.~H.~Bauschke), hui.ouyang@alumni.ubc.ca (H.~Ouyang), shawn.wang@ubc.ca (X.~Wang)
\par
Received May 11, 2021; Accepted  TBA. }

\section{Introduction}

Throughout this paper, we assume that
\begin{empheq}[box = \mybluebox]{equation*}
\text{$\mathcal{H}$ is a real Hilbert space},
\end{empheq}
with inner product $\innp{\cdot,\cdot}$ and induced norm $\|\cdot\|$.

Following Deutsch and Hundal \cite{DH2006II}, we now recall
two notions of angles between two nonempty convex sets:

\begin{definition} {\bf (Dixmier and Friedrichs angle)} {\rm \cite[Definitions~2.3 and 3.2]{DH2006II}} \label{defn:Angles}
	Let $C$ and $D$ be nonempty convex sets in $\mathcal{H}$.
	The \emph{minimal angle} or \emph{Dixmier angle} between $C$ and $D$ is the angle in
	$\left[0, \frac{\pi}{2} \right]$ whose cosine is given by
	\begin{empheq}[box = \mybluebox]{equation}
	\label{eq:minimalangle}
	c_{0} (C,D) := \sup \big\{   \innp{x,y} ~\big| ~ x \in \overline{\cone} (C) \cap \mathbf{B}_{\mathcal{H}}, ~ y \in \overline{\cone} (D) \cap \mathbf{B}_{\mathcal{H}}\big\},
	\end{empheq}
	where the $\cone$ and $\overline{\cone}$ operator returns
	the smallest cone and closed cone containing its argument, respectively, and
	where $\mathbf{B}_{\mathcal{H}}$ denotes the closed unit ball in $\mathcal{H}$.
	In addition, the \emph{angle} or \emph{Friedrichs angle}
	between $C$ and $D$ is the angle in $\left[0, \frac{\pi}{2} \right]$ whose cosine is given by
	\begin{empheq}[box = \mybluebox]{equation}
	c(C,D) := c_{0} \big( (\cone C) \cap \overline{ C^{\ominus} +D^{\ominus}},
	~ (\cone D) \cap  \overline{ C^{\ominus} +D^{\ominus}} \big),
	\end{empheq}
	and where $C^\ominus := \{ y\in\mathcal{H}~|~ \sup\innp{C,y}\leq 0\}$ is the polar cone of $C$
	and $C^\oplus := -C^\ominus$ is the dual cone of $C$.
\end{definition}
The cosine of these angles plays a key role in describing convergence rates
for projection methods such as the standard cyclic projection algorithm.
For a taste of such results, we refer the reader to, e.g.,
\cite{DH2006II}, \cite{DH2008III},
and \cite[Example~9.40]{D2012}.

When $M$ and $N$ are closed linear subspaces of $\mathcal{H}$,
then the cosines of their Dixmier and Friedrichs angles --- which were first studied
in \cite{Dixmier1949} and \cite{Friedrichs1937} --- can be written more succinctly as
\begin{equation*}
c_{0} (M,N) = \sup \big\{   \innp{x,y} ~\big| ~ x \in M \cap \mathbf{B}_{\mathcal{H}}, ~ y \in N \cap \mathbf{B}_{\mathcal{H}}\big\}
\end{equation*}
and
\begin{equation*}
c(M,N) := c_{0} \big( M \cap (M\cap N)^\perp,
~ N \cap  (M\cap N)^\perp \big),
\end{equation*}
respectively.
The following three key results in this linear setting
were beautifully presented and proved in Deutsch's survey \cite{Deutsch1995} to which
we also refer the reader for interesting historical comments:

\begin{fact} {\bf (Hundal's Lemma)} {\rm (see \cite[Lemma~14]{Deutsch1995})} \label{f:Hundal}
	Let $M$ and $N$ be closed linear subspaces of $\mathcal{H}$ such that
	$c_{0} (M,N) <1$, and  let $X$ be a closed linear subspace such that $M+N\subseteq X \subseteq
	\mathcal{H}$. Then
	\begin{equation*}
	c_{0} (M,N) \leq c_{0} (M^{\perp} \cap X, N^{\perp} \cap X ).
	\end{equation*}
\end{fact}

\begin{fact} {\bf (Krein, Krasnoselskii, and Milman)} {\rm (see \cite[Theorem~15]{Deutsch1995})} \label{f:KKM}
	Let $M$ and $N$ be closed linear subspaces of $\mathcal{H}$ such that
	$M\cap N = \{0\}$ and $M+N=\mathcal{H}$. Then
	\begin{equation*}
	c_{0} (M,N) = c_0(M^\perp,N^\perp).
	\end{equation*}
\end{fact}

\begin{fact} {\bf (Solmon)} {\rm (see \cite[Theorem~16]{Deutsch1995})} \label{f:Solmon}
	Let $M$ and $N$ be closed linear subspaces of $\mathcal{H}$.
	Then
	\begin{equation*}
	c(M,N) = c(M^\perp,N^\perp).
	\end{equation*}
\end{fact}

We are now in a position to describe \emph{the aim of this paper:}
{\em We will study the possibility of generalizing \cref{f:Hundal}, \cref{f:KKM}, and \cref{f:Solmon}
	from linear subspaces to convex cones.}
While we initially anticipated nice generalizations, it turned out that only
\cref{f:Hundal} appears to admit a natural and nice generalization.
The situation is more complicated for \cref{f:KKM} and \cref{f:Solmon};
our results and examples show that there are ostensibly no ``nice'' conical variants.
Our analysis relies on Deutsch's exposition \cite{Deutsch1995} as well as
our recent work \cite{BOWJMAA}.

The remainder of the paper is organized as follows.
In \cref{sec:aux}, we collect facts useful in subsequent proofs.
The (positive) results concerning Hundal's Lemma (\cref{f:Hundal}) are
presented in \cref{sec:pos}.
The (somewhat negative) results concerning \cref{f:KKM} and \cref{f:Solmon}
are provided in \cref{sec:neg}.

Finally, the notation we employ is standard and follows
\cite{BC2017} and \cite{D2012}.
For instance, orthogonality of vectors and sets as well as 
orthogonal complements are indicated by 
``$\perp$'' while ``$+$'' denotes the sum (not necessarily orthogonal)
of vectors and sets. If $S\subseteq \mathcal{H}$, then 
the smallest convex cone (resp.\ closed convex cone)
containing $S$ is denoted by $\cone(S)$ (resp.\ $\overline{\cone}(S)$).

\section{Auxiliary result}
\label{sec:aux}

In this section, we simply list --- for the reader's convenience ---
several known results that are used in proving our new results
in \cref{sec:pos} and \cref{sec:neg}.

\begin{fact} {\rm \cite[Theorem~4.5]{D2012}} \label{fact:dualcone}
	Let $C$ be a nonempty subset of  $\mathcal{H}$.  Then the following hold:
	\begin{enumerate}
		\item  \label{fact:dualcone:o} $C^{\ominus}$ is a closed convex cone and $C^{\perp}$ is a closed linear subspace.
		\item   \label{fact:dualcone:eq}$C^{\ominus} =(\overline{C})^{\ominus}=\left( \cone(C) \right)^{\ominus}=\left( \overline{\cone}(C) \right)^{\ominus}$.
		\item \label{fact:dualcone:oo} $ C^{\ominus\ominus} = \overline{\cone}(C) $.
		
		\item  \label{fact:dualcone:closedconvex} If $C$ is a closed convex cone, then $C^{\ominus\ominus} =C$.
		\item  \label{fact:dualcone:linearsubsp} If $C$ is a linear subspace,
		then $C^{\ominus}=C^{\perp}$; if $C$ is additionally closed, then $C=C^{\ominus\ominus}=C^{\perp \perp}$.
	\end{enumerate}
\end{fact}

\begin{fact} {\rm \cite[Lemma~2.5]{BOWJMAA}} \label{lemma:polardual}
	Let $C$  be a nonempty subset  of $\mathcal{H}$.   Then the following statements hold:
	\begin{enumerate}
		\item \label{lemma:polardual:-C} $	(-C)^{\ominus} =- C^{\ominus} =C^{\oplus}$.
		\item \label{lemma:polardual:oplus} $C^{\oplus\oplus}=\overline{\cone} (C)$.
		
		\item  \label{lemma:polardual:llinear} If $C$ is a linear subspace of $\mathcal{H}$, then $C^{\perp}=C^{\ominus}=C^{\oplus}$.
	\end{enumerate}	
\end{fact}

\begin{fact} {\rm \cite[Propositions~6.3 and 6.4]{BC2017} } \label{lemma:K}
	Let $K$ be a nonempty convex cone in $\mathcal{H}$.
	Then  $K+K=K$.
	If $-K \subseteq K$, then $K$ is a   linear subspace.
\end{fact}

\begin{fact}  {\rm \cite[Propositions~6.28]{BC2017}} \label{fact:chara:PK}
	Let $K$ be a nonempty closed convex cone in $\mathcal{H}$, let $x \in \mathcal{H}$, and let
	$p \in \mathcal{H}$. Then $ p =\Pro_{K} x \Leftrightarrow \left[  p \in K, x-p \perp p,  ~ \text{and}~ x-p \in K^{\ominus}\right]$.
\end{fact}

\begin{fact} {\rm  \cite[page~48]{SS2016I} } \label{fact:coBS}
	Let $C$ and $D$ be nonempty convex subsets of $\mathcal{H}$
	such that $C \neq \{0\}$ and $D \neq \{0\}$.  Then
	\begin{align*}
	c_{0} (C,D) = \max \Big\{ 0, \sup \big\{   \innp{x,y} ~\big|~ x \in \overline{\cone} (C) \cap \mathbf{S}_{\mathcal{H}}, ~ y \in \overline{\cone} (D) \cap \mathbf{S}_{\mathcal{H}}\big\}   \Big\},
	\end{align*}
	where $\mathbf{S}_{\mathcal{H}}$ denotes the unit sphere in $\mathcal{H}$. 
\end{fact}

\begin{fact} {\rm \cite[Lemma~2.4, Theorem~2.5 and Proposition~3.3]{DH2006II}}
	\label{fact:AnglesProperties}
	Let $C$ and $D$ be   nonempty convex subsets of $\mathcal{H}$.
	Then the following hold:
	\begin{enumerate}
		\item \label{fact:AnglesProperties:between0and1}$ c_{0} (C,D)  \in \left[0,1\right]$ and $ c (C,D)  \in \left[0,1\right]$.
		\item \label{fact:AnglesProperties:ineq} $(\forall x \in \overline{\cone} (C) )$ $(\forall y \in \overline{\cone} (D) )$ $\innp{x,y} \leq   c_{0} (C,D) \norm{x} \norm{y}$.
		\item  \label{fact:AnglesProperties:c0:EQ} $ c_{0} (C,D)=c_{0} (D, C)=c_{0} \left( \overline{ C} ,\overline{D} \right)=c_{0} \left(  \cone  (C) , \cone (D) \right)=c_{0} \left( \overline{\cone} (C) ,\overline{\cone} (D) \right)$.
	\end{enumerate}
\end{fact}

\begin{fact} {\rm \cite[Lemma~2.11]{BOWJMAA}}\label{lemma:CD}
	Let $C$ and $D$ be   nonempty convex subsets of $\mathcal{H}$. Then the following hold:
	\begin{enumerate}
		\item \label{lemma:CD:UV} If $U$ and $V$ are nonempty convex subsets of
		$\mathcal{H}$ such that $C \subseteq U$ and $D \subseteq V$, then $c_{0}(C,D) \leq c_{0}(U,V)$.
		\item \label{lemma:CD:--}  $c_{0}(C,D) = c_{0}(-C,-D)$, $c_{0}(-C,D) = c_{0}(C,-D)$, $c (C,D) = c (-C,-D)$, and  $c (-C,D) = c (C,-D)$.
		
		\item  \label{lemma:CD:C0=1}
		If $ (\overline{\cone} (C) \cap \overline{\cone} (D) ) \smallsetminus \{0\} \neq \varnothing$, then $c_{0} (C,D)  =1$.
		
		\item \label{lemma:CD:c0:c} $0 \leq c(C,D) \leq c_{0} (C,D) \leq 1$.
	\end{enumerate}
\end{fact}

\begin{fact} {\rm \cite[Theorem~4.6]{D2012}} \label{fact:K1:Km:dualsum}
	Let $K_{1}$ and $K_2$ be nonempty closed convex cones in $\mathcal{H}$.
	Then
	\begin{equation*}
	(K_1\cap K_2)^\ominus = \overline{K^{\ominus}_{1} + K_2^\ominus}.
	\end{equation*}
\end{fact}

\begin{fact} {\rm \cite[Propositions~3.3(4)]{DH2006II}} \label{fact:coneK1K2}
	Let  $K_{1}$ and $K_{2}$ be  nonempty  closed convex cones in $\mathcal{H}$.
	Then
	\begin{equation*}
	c(K_{1},K_{2}) = c_{0} \left(  K_{1}  \cap (K_{1}\cap K_{2})^{\ominus}, K_{2} \cap (K_{1}\cap K_{2})^{\ominus} \right).
	\end{equation*}
\end{fact}

\begin{fact} {\rm \cite[Lemma~4.1]{BOWJMAA}}\label{lemma:cK1K2}
	Let $K_{1}$ and $K_{2}$ be nonempty closed convex cones in $\mathcal{H}$. Then the following hold:
	\begin{enumerate}
		\item \label{lemma:cK1K2:1} If $K_{1} \cap K_{2} \neq \{0\}$, then $c_{0}(K_{1},K_{2}) =1$.
		\item \label{lemma:cK1K2:0EQ} If $K_{1} \cap K_{2} =\{0\}$,  then $c_{0}(K_{1},K_{2}) =c(K_{1},K_{2}) $.

		\item \label{lemma:cK1K2:H}  $K_{1} \cap K_{2} =\{0\}$ if and only if $ \overline{K_{1}^{\ominus}+K_{2}^{\ominus}} =\mathcal{H}$.
		
	\end{enumerate}
\end{fact}

\begin{fact} {\rm \cite[Proposition~4.4]{BOWJMAA}}\label{prop:cK1K2}
	Let $\mathcal{K}$ be a finite-dimensional linear subspace of $\mathcal{H}$, and
	let $K_{1}$ and $K_{2}$ be nonempty closed convex cones in $\mathcal{H}$ such that
	$K_{1} \subseteq  \mathcal{K}$ or $K_{2} \subseteq  \mathcal{K}$.
	Then the following hold:
	\begin{enumerate}
		\item \label{prop:cK1K2:c01} $K_{1} \cap K_{2} \neq \{0\}$ if and only if  $c_{0}(K_{1},K_{2}) =1$.
		\item \label{prop:cK1K2:EQ}   $K_{1} \cap K_{2} =\{0\}$ if and only if  $c_{0}(K_{1},K_{2}) =c(K_{1},K_{2}) $.
		\item \label{prop:cK1K2:c01:EQ} $K_{1} \cap K_{2} = \{0\}$ if and only if  $c_{0}(K_{1},K_{2}) <1$.
	\end{enumerate}
\end{fact}

\begin{fact} {\rm \cite[Theorem~4.7]{BOWJMAA}}\label{theorem:K1K2closed}	
	Let $K_{1}$ and $K_{2}$ be  nonempty closed convex cones in
	$\mathcal{H}$. If $c_{0}(K_{1}, K_{2}) <1$,
	then $K_{1} - K_{2}$ is closed.
\end{fact}

\begin{fact} {\rm \cite[Theorem~4.11]{BOWJMAA}}\label{theorem:KominusMperpNeq0}
	Let $K_{1} $ and $K_{2}$  be  closed convex cones in $\mathcal{H}$
	such that $K_{1} \cap K_{2}=\{0\}$ and $K_1$ is not linear.
	Furthermore, suppose that one of the following holds:
	\begin{enumerate}
		\item \label{theorem:KominusMperpNeq0:M} There exists $u \in \mathcal{H}$ such that $K_{2} =\{u  \}^{\perp}$.
		\item \label{theorem:KominusMperpNeq0:H} There exists $u \in \mathcal{H}$
		such that  $K_{2} \subseteq \{u\}^\perp$ and $\{u\}^\perp \cap K_{1} =\{0\}$.
		\item \label{theorem:KominusMperpNeq0:K}
		There exists a finite-dimensional linear subspace $\mathcal{K}$ of $\mathcal{H}$
		such that $K_{1} \subseteq  \mathcal{K}$ or $K_{2} \subseteq  \mathcal{K}$.
	\end{enumerate}
	Then
	\begin{align*}
	K^{\ominus}_{1} \cap K^{\oplus}_{2} \neq \{0\} \quad \text{and} \quad K^{\oplus}_{1} \cap K^{\ominus}_{2} \neq \{0\}.
	\end{align*}
\end{fact}

\section{Positive results}

\label{sec:pos}

We begin with a simple generalization of
\cite[Lemma~10.2]{Deutsch1995} from two linear subspaces to one cone and one
linear subspace.

\begin{lemma} \label{lemma:MK}
	Let $K  $ be a nonempty closed convex cone in $\mathcal{H} $ and let $M   $ be a
	closed linear subspace of $\mathcal{H}$. Then
	\begin{align*}
	(\forall x \in K) (\forall  y \in M) \quad  \abs{\innp{x,y}} \leq c_{0}(K,M) \norm{x}\norm{y}.
	\end{align*}
\end{lemma}
\begin{proof}
	Let $x \in K$ and $ y \in M$. Then $-y\in M$ as well because $M$ is linear.
	Applying
	\cref{fact:AnglesProperties}\cref{fact:AnglesProperties:ineq} with $C=M$ and
	$D=K$ to obtain that
	$\pm\innp{x,y}=	\innp{x,\pm y} \leq c_{0}(K,M) \norm{x}\norm{\pm y}
	=c_{0}(K,M) \norm{x}\norm{y}$
	and the result follows.
\end{proof}

When one cone is contained in another, then the following pleasing result holds:

\begin{proposition} \label{prop:K1K1c=0}
	Let $K_{1}$ and $K_{2}$ be nonempty closed convex cones in $\mathcal{H}$
	such that $K_{1} \subseteq K_{2}$.
	Then
	\begin{equation*}
	c(K_{1}, K_{2}) =0=c(K^{\ominus}_{1}, K^{\ominus}_{2}).
	\end{equation*}
\end{proposition}
\begin{proof}
	Because $K_{1} \subseteq K_{2}$, we have $K_{1} \cap K_{2}=K_{1}$ and $(K_{1} \cap
	K_{2})^{\ominus}=K_{1}^{\ominus}$. Now $K_{1}$ is a nonempty closed
	cone and so $0 \in K_{1}$. Hence, by \cref{lemma:K}, we have that $K_{2}=
	\{0\} +K_{2} \subseteq K_{1}+K_{2} \subseteq K_{2} +K_{2} =K_{2}$,
	which implies that $K_{1}+K_{2}= K_{2}$.
	Combine this with \cref{fact:K1:Km:dualsum} and  \cref{fact:dualcone}\cref{fact:dualcone:closedconvex} to see that $(K_{1}^{\ominus} \cap K_{2}^{\ominus})^{\ominus} =\overline{K_{1}^{\ominus\ominus} +K_{2}^{\ominus\ominus} } =\overline{K_{1} +K_{2}}=\overline{K_{2}}=K_{2}$.
	Moreover, using \cref{fact:coneK1K2} and \cref{defn:Angles}, we obtain
	\begin{align*}
	c(K_{1},K_{2}) &= c_{0} \left(  K_{1}  \cap (K_{1}\cap K_{2})^{\ominus}, K_{2} \cap (K_{1}\cap K_{2})^{\ominus} \right) = c_{0} \left(  K_{1}  \cap  K_{1}^{\ominus}, K_{2} \cap  K_{1}^{\ominus} \right)\\
	&= c_{0} \left(  \{0\}, K_{2} \cap K_{1}^{\ominus} \right)=0
	\end{align*}
	and
	\begin{align*}
	c(K_{1}^{\ominus},K_{2}^{\ominus}) &= c_{0} \left(  K_{1}^{\ominus}  \cap (K_{1}^{\ominus} \cap K_{2}^{\ominus})^{\ominus}, K_{2}^{\ominus} \cap (K_{1}^{\ominus} \cap K_{2}^{\ominus})^{\ominus} \right)
	%= c_{0} \left(  K_{1}^{\ominus}  \cap K_{2}, K_{2}^{\ominus} \cap K_{2} \right)
	= c_{0} \left(  K_{1}^{\ominus} \cap K_{2}, \{0\}\right)=0,
	\end{align*}
	which completes the proof.
\end{proof}

We are now ready for our first main result --- the concial extension of
Hundal's Lemma (\cref{f:Hundal}):

\begin{theorem} {\bf (concial extension of Hundal's Lemma)} \label{prop:KMoplus}
	Let $ K_{1}$ and $ K_{2}$  be nonempty closed convex cones in $ \mathcal{H}$.
	Suppose that $c_{0} (K_{1},K_{2}) <1$ and that $X$ is a convex subset of $\mathcal{H}$ which contains $K_{1}-K_{2}$.  Then
	\begin{align*}
	c_{0} (K_{1}, K_{2})  \leq c_{0} (K_{1}^{\oplus} \cap X , K_{2}^{\ominus} \cap X ) \leq c_{0} (K_{1}^{\oplus}   , K_{2}^{\ominus} ).
	\end{align*}
\end{theorem}

\begin{proof}
	\cref{lemma:CD}\cref{lemma:CD:UV} yields
	\begin{align*}
	c_{0} (K_{1}^{\oplus} \cap (K_{1} -K_{2} ) , K_{2}^{\ominus} \cap (K_{1}-K_{2} ) )\leq c_{0} (K_{1}^{\oplus} \cap X , K_{2}^{\ominus} \cap X ) \leq c_{0} (K_{1}^{\oplus}, K_{2}^{\ominus} ).
	\end{align*}
	Hence, it suffices to prove that $c_{0} (K_{1}, K_{2})  \leq c_{0} (K_{1}^{\oplus} \cap (K_{1} -K_{2} ) , K_{2}^{\ominus} \cap (K_{1}-K_{2} ) )$.
	
	If $c_{0} (K_{1},K_{2}) =0$, then $c_{0} (K_{1}, K_{2})  \leq c_{0} (K_{1}^{\oplus} \cap (K_{1} -K_{2} ) , K_{2}^{\ominus} \cap (K_{1}-K_{2} ) )$ is trivial.
	
	Now assume that $c_{0} (K_{1},K_{2}) \in \left]0,1\right[\,$.
	Then, by \cref{fact:coBS}, there exist sequences $(x_{k})_{k \in \mathbb{N}}$ in $K_{1} \cap \mathbf{S}_{\mathcal{H}}$ and  $(y_{k})_{k \in \mathbb{N}}$ in $K_{2} \cap \mathbf{S}_{\mathcal{H}}$ such that $\innp{x_{k}, y_{k}} \to c_{0} (K_{1},K_{2})$.
	
	Clearly, $K_{1}^{\oplus}=-K_{1}^{\ominus} $.
	Moreover, by \cref{fact:chara:PK},  $(\forall k \in \mathbb{N})$ $x_{k} -
	\Pro_{K_{2}}x_{k} \in K^{\ominus}_{2} \cap (K_{1} -K_{2}) $ and
	$\Pro_{K_{1}}\Pro_{K_{2}}x_{k}-\Pro_{K_{2}}x_{k} \in K_{1}^{\oplus} \cap (K_{1}
	-K_{2})$.
	Set
	$c_{0}:=c_{0} (K_{1}, K_{2})$ and
	$(\forall k \in \mathbb{N})$ $
	\alpha_{k}:=\norm{\Pro_{K_{2}}x_{k} } $.
	Hence, using \cref{lemma:CD}\cref{lemma:CD:UV}, \cref{fact:chara:PK},
	and the same techniques employed
	in the proof of \cite[Lemma~14]{Deutsch1995}, we obtain that $\alpha_{k} \to c_{0}$
	and that for every $k \in \mathbb{N}$,
	\begin{align*}
	c_{0} (K_{1}^{\oplus} \cap (K_{1} -K_{2} ) , K_{2}^{\ominus} \cap (K_{1}-K_{2} ) )
	&\geq   \frac{\biginnp{ \Pro_{K_{1}}\Pro_{K_{2}}x_{k}-\Pro_{K_{2}}x_{k} ,   x_{k} - \Pro_{K_{2}}x_{k} } }{\norm{\Pro_{K_{1}}\Pro_{K_{2}}x_{k}-\Pro_{K_{2}}x_{k}}  \norm{x_{k} - \Pro_{K_{2}}x_{k}}}\\
	&\geq   \alpha_{k}- \sqrt{ \frac{c^{2}_{0} -\alpha^{2}_{k} }{ 1-c^{2}_{0}}} \to c_{0},
	\end{align*}
	which completes the proof.
\end{proof}

\begin{remark}
	If $K_1$ and $K_2$ are linear in \cref{prop:KMoplus},
	then $K_1-K_2=K_1+K_2$, $K_1^\oplus = K_1^\perp$,
	$K_2^\ominus = K_2^\perp$, and we recover \cref{f:Hundal}.
\end{remark}

We conclude this section with further comments on \cref{{prop:KMoplus}} which are based
on the following example.

\begin{example}\label{exam:K1K2} {\rm \cite[Example~4.13]{BOWJMAA}}
	Suppose that $\mathcal{H} =\mathbb{R}^{2}$.
	Set $K_{1}:=\mathbb{R}^{2}_{+}$ and $K_{2}:= \{ (x_{1}, x_{2}) \in \mathbb{R}^{2} ~|~ -x_{1} \geq x_{2}   \}$.
	Then the following hold:
	\begin{enumerate}
		\item \label{exam:K1K2:ominus} $K_{1}^{\ominus}= \mathbb{R}^{2}_{-}$, $K_{2}^{\ominus}=\mathbb{R}_{+} (1,1)$, and $K_{2}^{\oplus} =\mathbb{R}_{+}(-1,-1)$.
		\item  \label{exam:K1K2:cap} $K_{1} \cap K_{2} =\{0\}$,  $K_{1}^{\ominus} \cap K_{2}^{\ominus} =\{0\}$, $K_{1}^{\ominus} \cap K_{2}^{\oplus} =K_{2}^{\oplus}$, $K_{1} +K_{2} =\mathbb{R}^{2}$, and $K_{1} - K_{2} =- K_{2}\neq\mathcal{H}$.
		\item \label{exam:K1K2:ineq} $c  (K_{1}, K_{2}) =c_{0} (K_{1}, K_{2}) =\tfrac{1}{\sqrt{2}} > 0=c_{0}  (K_{1}^{\ominus}, K_{2}^{\ominus}) =c_{0}  (K_{1}^{\oplus}, K_{2}^{\oplus}) =c (K_{1}^{\ominus}, K_{2}^{\ominus}) =c  (K_{1}^{\oplus}, K_{2}^{\oplus})$.
		
		\item \label{exam:K1K2:ineqoplus} $c_{0} (K_{1}^{\ominus},K_{2}^{\oplus}) =1$, $c  (K_{1}^{\ominus},K_{2}^{\oplus}) =0$, $c_{0} (K_{1},K_{2}) <  c_{0} (K_{1}^{\ominus},K_{2}^{\oplus}) $, and  $c  (K_{1},K_{2}) > c  (K_{1}^{\ominus},K_{2}^{\oplus}) $.
	\end{enumerate}
\end{example}

\begin{remark} {\rm \bf (importance of using dual-polar pair)}
	Let $K_1$ and $K_2$ be as in \cref{exam:K1K2}.
	Then $c_{0} (K_{1}, K_{2}) <1$ yet
	\begin{equation*}
	c_0(K_1,K_2) = \tfrac{1}{\sqrt{2}}
	>
	0 = c_0(K_1^\ominus,K_2^\ominus) = c_0(K_1^\oplus,K_1^\oplus).
	\end{equation*}
	This illustrates that in the conical extension of Hundal's Lemma
	(\cref{prop:KMoplus}), we indeed must work with a pair of
	dual-polar cones of the original pair --- substituting with
	a polar-polar or dual-dual pair will not work!
\end{remark}

\section{Negative results}

\label{sec:neg}
We start with characterizations for differences of convex cones to be linear subspaces.

\begin{lemma}\label{l:difference}
	Let $K_{1}$ and $K_{2}$ be nonempty closed convex cones in $\mathcal{H}$ such that  
	$K_{1}\cap K_{2}=\{0\}$. Then the following are equivalent:
	\begin{enumerate}
		\item\label{i:subspace} $K_{1}-K_{2}$ is a linear subspace of $\mathcal{H}$.
		\item\label{i:superset} $(-K_1)\cup K_2 \subseteq K_{1}-K_{2}$. 
		\item\label{i:bothsub} $K_{1}$ and $K_{2}$ are linear subspaces of $\mathcal{H}$.
	\end{enumerate}
\end{lemma}

\begin{proof}
	``\cref{i:subspace}$\Rightarrow$\cref{i:superset}'': 
	We have 
	$K_1 = K_1 - \{0\} \subseteq K_1-K_2$ and 
	$-K_2 = \{0\}-K_2 \subseteq K_1-K_2$.
	Hence 
	$(-K_1)\cup K_2 = -(K_1\cup (-K_2)) \subseteq -(K_1-K_2)=K_1-K_2$.

	``\cref{i:superset}$\Rightarrow$\cref{i:bothsub}'': 
	Let $x\in -K_{1}$. Then there exist $y_1\in K_1$ and $y_2\in K_2$ such that 
	$x=y_1-y_2$.
	Then $y_2 = y_1-x \in K_1 - (-K_1) = K_1 + K_1 = K_1$.
	Hence $y_2 \in K_1\cap K_2 = \{0\}$. 
	Thus $y_2 = 0$ and so $x = y_1\in K_1$. 
	This implies $-K_1\subseteq K_1$ and therefore
	$K_1$ is linear by \cref{lemma:K}. 
	The argument for $K_2$ is similar. 
	
	``\cref{i:bothsub}$\Rightarrow$\cref{i:subspace}'': Obvious.
\end{proof}

Here is a conical variant of \cref{f:KKM}.

\begin{theorem} \label{theorem:K1K2ominusoplusEQ}
	Let $ K_{1}$ and $ K_{2}$  be  nonempty closed convex cones in $ \mathcal{H}$
	such that $K_{1} \cap K_{2} =\{0\}$ and $K_1-K_2=\mathcal{H}$.
	Then
	\begin{enumerate}
		\item  \label{theorem:K1K2ominusoplusEQ:linear} $K_{1}$ and $K_{2}$ are linear subspaces
		of $\mathcal{H}$. 
		\item  \label{theorem:K1K2ominusoplusEQ:eq}
		$c_{0} (K_{1}, K_{2})  =  c_{0} (K_{1}^{\oplus}   , K_{2}^{\ominus}   )= c_{0} (K_{1}^{\ominus}   , K_{2}^{\oplus}   )=
		c_{0}(K_{1}^{\perp}, K_{2}^{\perp}).$
	\end{enumerate}
\end{theorem}

\begin{proof}
	\cref{theorem:K1K2ominusoplusEQ:linear}:
	Clear from \cref{l:difference}. 
	\cref{theorem:K1K2ominusoplusEQ:eq}: 
	In view of \cref{theorem:K1K2ominusoplusEQ:linear},
	we have $-K_2=K_2$ and thus $K_1+K_2=K_1-K_2=\mathcal{H}$. 
	Now apply \cref{f:KKM}.
\end{proof}

\begin{remark}{\bf (impossibility of extending \cref{f:KKM} to cones)}
	The assumption on the two subspaces in \cref{f:KKM} is that
	their intersection is $\{0\}$ and their sum is the entire space.
	Now assume that $K_2$ is a linear subspace
	and $K_1+K_2=K_1-K_2$.
	Then Lemma~\ref{l:difference} implies that $K_1$ is a subspace as well.
	Hence \cref{f:KKM} does not appear to admit an extension to the case
	when one linear subspace is replaced by a convex cone!
\end{remark}

\begin{remark}{\bf (importance of the assumption that $K_1-K_2=\mathcal{H}$)}
	Let $K_1$ and $K_2$ be as in \cref{exam:K1K2}.
	Then $K_1\cap K_2 = \{0\}$ but neither $K_1$ nor $K_2$ is a linear subspace.
	This shows that the assumption
	that $\mathcal{H} = K_{1}-K_{2}$  in \cref{theorem:K1K2ominusoplusEQ}
	%and \cref{cor:K1K2ominusoplusEQ}
	is critical.
\end{remark}

\begin{remark}{\bf (importance of the assumption that $K_1\cap K_2=\{0\}$)}
	Let $u_{1}, u_{2}$ be unit vectors in $\mathbb{R}^3$ such that $0<\scal{u_{1}}{u_{2}}<1$. 
	Set $K_1:=\{u_{1}\}^{\perp}$ and $K_2=\{u_{2}\}^{\perp}$. Then 
	$K_{1}^{\perp}=\mathbb{R}u_{1}$, 
	$K_{2}^{\perp}=\mathbb{R}u_{2}$, 
	$K_1\cap K_2\neq \{0\}$ and $K_{1}-K_{2}=\mathbb{R}^{3}$, but 
	$c_{0}(K_{1}^{\perp},K_{2}^{\perp})=\scal{u_{1}}{u_{2}}<1=c_{0}(K_{1},K_{2})$.
\end{remark}

We now turn to a conical variant of \cref{f:Solmon}.
Our proof is an adaptation of Deutsch's proof of \cite[Theorem~16]{Deutsch1995}.

\begin{theorem} \label{theorem:cEQ}
	Let $K_{1}$ and $K_{2}$ be nonempty closed convex cones in $\mathcal{H}$.
	Suppose that $c(K_{1},K_{2}) <1$, that $K_{1} \cap K_{2}$  and
	$K_{1}^{\oplus} \cap K_{2}^{\ominus}$  are linear subspaces of
	$\mathcal{H}$, that $\overline{K_{1}^{\oplus} +(K_{1} \cap K_{2})} \cap
	(K_{1} \cap K_{2})^{\perp} =K_{1}^{\oplus}$, and that $\overline{K_{2}^{\ominus}
		+(K_{1} \cap K_{2})} \cap (K_{1} \cap K_{2})^{\perp} =K_{2}^{\ominus}$. Then
	\begin{align*}
	c(K_{1}, K_{2})  \leq  c(K_{1}^{\oplus}, K_{2}^{\ominus}).
	\end{align*}
	If additionally $c(K_{1}^{\oplus},K_{2}^{\ominus}) <1$,
	\begin{align*}
	\overline{K_{1} +(K_{1}^{\oplus} \cap K_{2}^{\ominus})}   \cap (K_{1}^{\oplus} \cap K_{2}^{\ominus})^{\perp}=K_{1}, \quad \text{and} \quad \overline{K_{2} +(K_{1}^{\oplus} \cap K_{2}^{\ominus})} \cap (K_{1}^{\oplus} \cap K_{2}^{\ominus})^{\perp}=K_{2},
	\end{align*}
	then
	\begin{align*}
	c(K_{1}, K_{2})  =  c(K_{1}^{\oplus}, K_{2}^{\ominus}).
	\end{align*}
\end{theorem}

\begin{proof} % Set $X:= (K_{1} \cap K_{2})^{\perp} \cap (K_{1}^{\oplus} \cap   K_{2}^{\ominus})^{\perp}$.
	\cref{fact:coneK1K2} and \cref{fact:dualcone}\cref{fact:dualcone:o} imply that
	\begin{align}\label{eq:prop:cEQ:leq:c0}
	c_{0} \big(  K_{1}  \cap (K_{1}\cap K_{2})^{\perp}, K_{2} \cap (K_{1}\cap K_{2})^{\perp} \big) =c(K_{1},K_{2}) <1.
	\end{align}
	Because $K_{1}^{\oplus} \cap   K_{2}^{\ominus}$ is a linear subspace of
	$\mathcal{H}$, \cref{fact:K1:Km:dualsum},
	\cref{lemma:polardual}\cref{lemma:polardual:llinear} and
	\cref{fact:dualcone}\cref{fact:dualcone:o}\&\cref{fact:dualcone:linearsubsp} imply
	that  $(K_{1}^{\oplus} \cap   K_{2}^{\ominus})^{\perp} =\overline{
		K_{1}-  K_{2} }$ is a linear subspace of $\mathcal{H}$. Hence
	\begin{align} \label{eq:prop:cEQ:X}
	X:=  (K_{1} \cap K_{2})^{\perp} \cap (K_{1}^{\oplus} \cap   K_{2}^{\ominus})^{\perp} =   (K_{1} \cap K_{2})^{\perp} \cap  \overline{ K_{1}-  K_{2} }
	\end{align}	
	is a closed linear subspace of $\mathcal{H}$.
	Because $\overline{ K_{1}-  K_{2} }$ is a closed linear subspace of $\mathcal{H}$,
	we learn that $K_{2}   \subseteq \overline{ K_{1}-  K_{2} }$
	and that $K_{1} \subseteq \overline{ K_{1}-  K_{2} }$.
	Hence
	\begin{align} \label{eq:prop:cEQ:KX}
	K_{1} \cap X =  K_{1} \cap (K_{1} \cap K_{2})^{\perp}
	\quad\text{and}\quad
	K_{2} \cap X =  K_{2} \cap (K_{1} \cap K_{2})^{\perp}.
	\end{align}
	Moreover,  $(K_{1} \cap K_{2})^{\perp}$ is a linear subspace and so $(K_{1}
	\cap K_{2})^{\perp} - (K_{1} \cap K_{2})^{\perp} = (K_{1} \cap
	K_{2})^{\perp}$.
	Combine this with \cref{eq:prop:cEQ:X} and
	\cref{eq:prop:cEQ:KX} to see that
	\begin{align} \label{eq:prop:cEQ:subseteqX}
	\big(K_{1} \cap (K_{1} \cap K_{2})^{\perp}\big)-
	\big(K_{2} \cap (K_{1} \cap K_{2})^{\perp}\big) =
	\big(K_{1} \cap X\big)  - \big(K_{2} \cap X\big) \subseteq X.
	\end{align}
	Using \cref{fact:K1:Km:dualsum},
	\cref{lemma:polardual}\cref{lemma:polardual:-C}$\&$\cref{lemma:polardual:oplus}
	and the assumption  that $\overline{K_{1}^{\oplus} +K_{1} \cap K_{2}} \cap (K_{1}
	\cap K_{2})^{\perp} =K_{1}^{\oplus}$,
	we obtain
	\begin{subequations}
		\label{eq:prop:cEQ:KXoplus}
		\begin{align}
		(K_{1} \cap (K_{1} \cap K_{2})^{\perp})^{\oplus} \cap X  &=  \overline{K_{1}^{\oplus} +K_{1} \cap K_{2}} \cap (K_{1} \cap K_{2})^{\perp} \cap (K_{1}^{\oplus} \cap   K_{2}^{\ominus})^{\perp}\\
		&=K_{1}^{\oplus}  \cap (K_{1}^{\oplus} \cap   K_{2}^{\ominus})^{\perp}.
		\end{align}
	\end{subequations}
	Similarly, using \cref{fact:K1:Km:dualsum},
	\cref{lemma:polardual}\cref{lemma:polardual:-C},
	\cref{fact:dualcone}\cref{fact:dualcone:closedconvex} and the assumption
	that $\overline{K_{2}^{\ominus} +K_{1} \cap K_{2}} \cap (K_{1} \cap
	K_{2})^{\perp} =K_{2}^{\ominus}$,
	we see that
	\begin{subequations}
		\label{eq:prop:cEQ:MXperp}
		\begin{align}
		(K_{2} \cap (K_{1} \cap K_{2})^{\perp})^{\ominus} \cap X &=  \overline{K_{2}^{\ominus} +K_{1} \cap K_{2}} \cap (K_{1} \cap K_{2})^{\perp} \cap (K_{1}^{\oplus} \cap   K_{2}^{\ominus})^{\perp} \\
		&=K_{2}^{\ominus}\cap(K_{1}^{\oplus} \cap   K_{2}^{\ominus})^{\perp}.
		\end{align}
	\end{subequations}
	We now use  \cref{eq:prop:cEQ:leq:c0}, \cref{eq:prop:cEQ:subseteqX}, and
	\cref{prop:KMoplus} (with $K_{1} = K_{1} \cap X$, $K_{2}=K_{2} \cap X$, and
	$X$ is as   \cref{eq:prop:cEQ:X}) to deduce that $c_{0} (  K_{1} \cap X ,
	K_{2} \cap X ) \leq c_{0} ((K_{1} \cap X)^{\oplus} \cap X,  (K_{2} \cap X
	)^{\ominus} \cap X)$, which, recalling \cref{eq:prop:cEQ:KX},
	\cref{eq:prop:cEQ:KXoplus}, and \cref{eq:prop:cEQ:MXperp}, is equivalent to
	\begin{align*}
	c_{0} (K_{1} \cap (K_{1} \cap K_{2})^{\perp} , K_{2} \cap (K_{1} \cap K_{2})^{\perp} ) \leq c_{0} \left(K_{1}^{\oplus}  \cap (K_{1}^{\oplus} \cap   K_{2}^{\ominus})^{\perp}, K_{2}^{\ominus}\cap(K_{1}^{\oplus} \cap   K_{2}^{\ominus})^{\perp}\right).
	\end{align*}
	This is	--- using \cref{fact:coneK1K2} and \cref{fact:dualcone}\cref{fact:dualcone:o} ---
	the same as $c(K_{1} ,K_{2} )  \leq c(K_{1}^{\oplus} , K_{2}^{\ominus})$.
	
	Finally, with the additional assumptions,
	apply the result just proved above  to
	$K_{1} =K_{1}^{\oplus}$ and $K_{2} =K_{2}^{\ominus}$ to conclude that
	$c(K_{1}^{\oplus} , K_{2}^{\ominus}) \leq  c(K_{1}^{\oplus \oplus} , K_{2}^{\ominus \ominus})  = c(K_{1},K_{2}) $, where the last equation
	follows from \cref{fact:dualcone}\cref{fact:dualcone:closedconvex}  and \cref{lemma:polardual}\cref{lemma:polardual:oplus}.
	Altogether, $c(K_{1},K_{2})  = c(K_{1}^{\oplus} , K_{2}^{\ominus})$ in this case.
\end{proof}

To derive some consequences of \cref{theorem:cEQ}, we require the following result.

\begin{lemma} \label{lemma:KMperp}
	Let $A$ and $B$ be nonempty subsets of  $\mathcal{H}$ such that
	$0 \in B$ and $A \subseteq B^{\perp}$.
	Then
	\begin{equation*}
	(A+B) \cap B^{\ominus} =A.
	\end{equation*}	
	Consequently, if $B$ is a linear subspace of $\mathcal{H}$,
	then $(A+B) \cap B^{\perp} =A$.
\end{lemma}
\begin{proof}
	Because  $0 \in B$ and
	$A \subseteq  B^{\ominus}$, we have that $ A \subseteq (A+B) \cap B^{\ominus}$.
	Conversely, let $x \in  (A+B) \cap B^{\ominus}$. Then there exist $a \in A$ and
	$b \in B$ such that  $x =a+b \in B^{\ominus}$.
	Because $x \in B^{\ominus}$,
	$b \in B$, and $a \in A \subseteq B^\perp = B^{\ominus} \cap B^{\oplus}$, we know that
	$\innp{x, b} \leq 0$, and that $\innp{a, b} =0$.
	Hence $0
	\geq \innp{x,b}=\innp{a+b,b}=\innp{a,b}+\norm{b}^{2} =\norm{b}^{2} \geq 0$
	so $\norm{b}^2 = 0$, i.e., $b=0$.
	Thus $x=a\in A$ and so
	$(A+B) \cap B^{\ominus} \subseteq A$.
	
	The ``Consequently'' part follows from what we just proved and
	\cref{lemma:polardual}\cref{lemma:polardual:llinear} which states that
	$B^{\perp}=B^{\ominus}=B^{\oplus}$ provided that $B$ is linear.
\end{proof}

We first reprove Solmon's results (\cref{f:Solmon}) from \cref{theorem:cEQ}.

%By \cref{cor:cEQ}, we know that \cref{theorem:cEQ} reduces \cite[Theorem~2.16]{Deutsch1995} when $K_{1}$ and $K_{2}$  are   linear subspaces of $\mathcal{H}$.
\begin{corollary}{\bf (Solmon)} \label{cor:cEQ}
	Let  $K_{1}$ and $K_{2}$ be closed linear subspaces of $\mathcal{H}$. Then
	$c(K_{1}, K_{2})  =  c(K_{1}^{\perp}, K_{2}^{\perp})$.	
\end{corollary}
\begin{proof}
	Because  $K_{1}$  and $K_{2}$ are closed linear subspaces,
	by \cite[Theorem~13]{Deutsch1995},
	\begin{align}\label{eq:cor:cEQ}
	c(K_{1},K_{2})   =1 \Leftrightarrow   K_{1} +K_{2} \text{ is not closed} \Leftrightarrow K_{1}^{\perp} +K_{2}^{\perp} \text{ is not closed}  \Leftrightarrow c(K_{1}^{\perp},K_{2}^{\perp}) =1.
	\end{align}
	Assume that $c(K_{1},K_{2})   <1$.
	By \cref{eq:cor:cEQ},
	$c(K_{1}^{\perp},K_{2}^{\perp})   <1$.  Note that
	$K_{1}^{\perp}=K_{1}^{\ominus}=K_{1}^{\oplus}$, and
	$K_{2}^{\perp}=K_{2}^{\ominus}=K_{2}^{\oplus}$.  Moreover, because
	$K_{1}^{\perp} \perp (K_{1} \cap K_{2})$, then by
	\cite[Proposition~29.6]{BC2017}, $\overline{K_{1}^{\perp} +(K_{1} \cap K_{2})}
	=K_{1}^{\perp} +K_{1} \cap K_{2}$. Hence, by \cref{lemma:KMperp},
	$\overline{K_{1}^{\oplus} +K_{1} \cap K_{2}} \cap (K_{1} \cap K_{2})^{\ominus} =
	(K_{1}^{\perp} +K_{1} \cap K_{2} )\cap (K_{1} \cap K_{2})^{\ominus}
	=K_{1}^{\perp}=K_{1}^{\oplus}$. Similarly,
	$\overline{K_{2}^{\ominus} -K_{1} \cap K_{2}} \cap (K_{1} \cap K_{2})^{\ominus}
	=K_{2}^{\ominus}$.
	Replacing $(K_{1},K_{2})$ with $(K_{1}^{\perp},K_{2}^{\perp})$ in the above,
	we obtain $ \overline{K_{1} +(K_{1}^{\oplus} \cap K_{2}^{\ominus})}   \cap (K_{1}^{\oplus}
	\cap K_{2}^{\ominus})^{\perp}=K_{1}$ and $ \overline{K_{2} -(K_{1}^{\oplus} \cap
		K_{2}^{\ominus})} \cap (K_{1}^{\oplus} \cap K_{2}^{\ominus})^{\perp}=K_{2}$.
	Hence, by \cref{theorem:cEQ}, $c(K_{1}, K_{2})  =  c(K_{1}^{\oplus},
	K_{2}^{\ominus}) =c(K_{1}^{\perp}, K_{2}^{\perp})$.
\end{proof}

Next we give some special cases of \cref{theorem:cEQ}.
\begin{corollary} \label{cor:prop:cEQ}
	Let $K_{1}$ and $K_{2}$ be nonempty closed convex cones in $\mathcal{H}$ such that
	$c(K_{1},K_{2}) <1$, $K_{1} \cap K_{2}$  and
	$K_{1}^{\oplus} \cap K_{2}^{\ominus}$  are linear subspaces of
	$\mathcal{H}$, and $ K_{1}^{\oplus} +(K_{1} \cap K_{2}) $  and $
	K_{2}^{\ominus} +(K_{1} \cap K_{2}) $ are closed. Then
	\begin{equation*}
	c(K_{1}, K_{2})  \leq  c(K_{1}^{\oplus}, K_{2}^{\ominus}).
	\end{equation*}
	If $c(K_{1}^{\oplus}, K_{2}^{\ominus}) <1$,  and $K_{1}
	+(K_{1}^{\oplus} \cap K_{2}^{\ominus})$  and $K_{2} +(K_{1}^{\oplus} \cap
	K_{2}^{\ominus}) $ are closed, then
	\begin{equation*}
	c(K_{1}, K_{2})  =  c(K_{1}^{\oplus}, K_{2}^{\ominus}).
	\end{equation*}
\end{corollary}
\begin{proof}
	Because $K_{1} \cap K_{2}$  and $K_{1}^{\oplus} \cap K_{2}^{\ominus}$  are linear subspaces of $\mathcal{H}$, by \cref{fact:dualcone}\cref{fact:dualcone:linearsubsp}, we have that $ K_{1}^{\oplus} \subseteq (K_{1} \cap K_{2})^{\perp} $, $ K_{2}^{\ominus} \subseteq (K_{1} \cap K_{2})^{\perp} $, $K_{1} \subseteq (K_{1}^{\oplus} \cap K_{2}^{\ominus})^{\perp}$  and $K_{2} \subseteq (K_{1}^{\oplus} \cap K_{2}^{\ominus})^{\perp}$.
	
	Hence, applying \cref{lemma:KMperp} with $A =K_{1}^{\oplus}$ and $B= K_{1} \cap
	K_{2}$, with  $A =K_{2}^{\ominus}$ and $B= K_{1} \cap K_{2}$, with $A=K_{1}$ and
	$ B= K_{1}^{\oplus} \cap K_{2}^{\ominus}$, and with $A=K_{2}$ and $ B=
	K_{1}^{\oplus} \cap K_{2}^{\ominus}$,  we obtain that  $(K_{1}^{\oplus} +(K_{1}
	\cap K_{2})) \cap (K_{1} \cap K_{2})^{\perp} =K_{1}^{\oplus}$,
	$(K_{2}^{\ominus}
	+(K_{1} \cap K_{2})) \cap (K_{1} \cap K_{2})^{\perp} =K_{2}^{\ominus}$, 	
	$(K_{1} +(K_{1}^{\oplus} \cap K_{2}^{\ominus}))   \cap (K_{1}^{\oplus} \cap
	K_{2}^{\ominus})^{\perp}=K_{1}$, and $(K_{2} +(K_{1}^{\oplus} \cap
	K_{2}^{\ominus})) \cap (K_{1}^{\oplus} \cap K_{2}^{\ominus})^{\perp}=K_{2}$,
	respectively.
	Therefore, the required results follow from the closedness assumptions and  \cref{theorem:cEQ}.
\end{proof}

\begin{corollary}  \label{propo:cKM:Rn}
	Let $K_{1}$ and $K_{2}$ be  nonempty closed convex cones in $\mathcal{H}$.
	Suppose that $K_{1}$ or $K_{2}$ is contained in a finite-dimensional
	linear subspace of $ \mathcal{H} $, and that  $K_{1}^{\oplus} $  or
	$K_{2}^{\ominus} $ is  contained in a finite-dimensional linear subspace of $
	\mathcal{H} $. Suppose furthermore that    $K_{1} \cap K_{2} =\{0\}$ and $K_{1}^{\oplus}
	\cap K_{2}^{\ominus}=\{0\}$.
	Then
	\begin{align*}
	c_{0}(K_{1},K_{2}) =	c(K_{1},K_{2}) =
	c(K_{1}^{\oplus},K_{2}^{\ominus})=c_{0}(K_{1}^{\oplus},K_{2}^{\ominus}).
	\end{align*}
\end{corollary}
\begin{proof}
	Because $K_{1} \cap K_{2} =\{0\}$, $K_{1}^{\oplus} \cap K_{2}^{\ominus}=\{0\}$,
	using the assumptions that $K_{1}$ or $K_{2}$ is contained in a
	finite-dimensional linear subspace of $ \mathcal{H} $, and that
	$K_{1}^{\oplus} $  or $K_{2}^{\ominus} $ is  contained in a finite-dimensional
	linear subspace of $ \mathcal{H} $,   by
	\cref{prop:cK1K2}, we have that
	\begin{align*}
	c (K_{1},K_{2}) =	c_{0}(K_{1},K_{2}) <1 \quad \text{and} \quad   c (K_{1}^{\oplus},K_{2}^{\ominus}) = c_{0}(K_{1}^{\oplus},K_{2}^{\ominus})<1.
	\end{align*}
	%	By \cref{eq:propo:cKM:Rn} and \cref{lemma:CD}\cref{lemma:CD:c0:c}, we know that $c (K_{1},K_{2}) <1 $ and  $ c (K_{1}^{\oplus},K_{2}^{\ominus})<1$.
	Hence, the desired identities are directly from \cref{cor:prop:cEQ}.
\end{proof}

Here is a simple illustration of  \cref{propo:cKM:Rn}.

\begin{example}
	\label{ex:210508a}
	Suppose that
	$\mathcal{H} =\mathbb{R}^{2}$, and
	set $K_{1} :=\mathbb{R}(1,0)$ and $K_{2} :=\mathbb{R}(1,1)$.
	Then
	\begin{equation*}
	c_{0}(K_{1},K_{2}) =	c(K_{1},K_{2}) =
	c(K_{1}^{\oplus},K_{2}^{\ominus})=c_{0}(K_{1}^{\oplus},K_{2}^{\ominus})
	=\tfrac{1}{\sqrt{2}}.
	\end{equation*}
	% In fact, by
	% \cref{theorem:KominusMperpNeq0}\cref{theorem:KominusMperpNeq0:K}, when $K_{1}$
	% or $K_{2}$ is contained in a finite-dimensional linear subspace of $ \mathcal{H}
	% $,   in order to satisfy that  $K_{1} \cap K_{2} =\{0\}$ and that
	% $K_{1}^{\oplus} \cap K_{2}^{\ominus} =\{0\}$, then both  $K_{1}$ and $K_{2}$
	% must be  closed linear subspaces.
\end{example}

The next example follows easily from the definitions and \cref{lemma:CD}\cref{lemma:CD:--}.
We will then comment on its relevance to previous results.

\begin{example} \label{example:KM}
	Suppose that $\mathcal{H} =\mathbb{R}^{2}$,
	and set  $K:=\{ (x_{1},x_{2}) \in \mathbb{R}^{2} ~|~ x_{2} \geq x_{1} \geq 0  \}$
	and $M:=\mathbb{R}(1,0)$.
	Then
	$K\cap M =\{0\}$, $K^{\ominus} =\{ (x_{1},x_{2}) \in \mathbb{R}^{2} ~|~
	-x_{1} \geq x_{2} \text{ and } x_{2} \leq 0 \}$, $M^{\perp}=\mathbb{R}(0,1)$,
	$K^{\ominus} \cap M^{\perp}= \mathbb{R}_{+}(0,-1)$,
	$(K^{\ominus} \cap M^{\perp} )^{\ominus} = \mathbb{R}\times\mathbb{R}_+$, and
	\begin{align*}
	0&= c(K^{\oplus},M^{\perp})= c(K^{\ominus},M^{\perp}) \\
	&<\tfrac{1}{\sqrt{2}} = c_{0}(K,M)=c(K,M)\\
	&<1 = c_{0}(K^{\oplus},M^{\perp})=c_{0}(K^{\ominus},M^{\perp}).
	\end{align*}
\end{example}

\begin{remark}{\bf (impossiblility of a simple extension of Solmon's Fact)}
	\cref{example:KM} very clearly shows the impossibility of
	a nice and simple generalization of \cref{f:Solmon} from
	two subspaces to even just a subspace and a cone:
	Indeed,
	in  \cref{example:KM}, we have
	\begin{equation*}
	c  (K^{\ominus}, M^{\perp}) =c (K^{\oplus}, M^{\perp})
	< c(K,M);
	\end{equation*}
	thus, neither the polar cone $K^\ominus$ nor the dual cone $K^\oplus$ will
	do the job!
\end{remark}

\begin{remark} {\bf (importance of linearity of $K_1^\oplus \cap K_2^\ominus$)}
	Let us revisit  \cref{theorem:cEQ} with
	$K_{1}$ and $K_{2}$  replaced by the $M$ and $K$
	from \cref{example:KM}, respectively.
	Then
	$c(K_1,K_2) = \tfrac{1}{\sqrt{2}} < 1$
	and
	$K_1\cap K_2 = \{0\}$.
	The latter clearly implies
	$\overline{K_{1}^{\oplus} +(K_{1} \cap K_{2})} \cap
	(K_{1} \cap K_{2})^{\perp} = K_{1}^{\oplus}$
	and
	$\overline{K_{2}^{\ominus}
		+(K_{1} \cap K_{2})} \cap (K_{1} \cap K_{2})^{\perp} =K_{2}^{\ominus}$.
	However,
	$K_1^\oplus \cap K_2^\ominus =  \mathbb{R}_+(0,-1)$ and so
	\begin{equation*}
	\text{$K_1^\oplus \cap K_2^\ominus$ is \emph{not} a linear subspace,
		and $c(K_1,K_2)>c(K_1^\oplus,K_2^\ominus)$.}
	\end{equation*}
	This shows that the assumption that $K_1^\oplus \cap K_2^\ominus$ be a linear subspace
	is critical in \cref{theorem:cEQ}.
	The very same example can be used to show the importance of the
	assumption that $K_1^\oplus \cap K_2^\ominus$ be a linear subspace
	in \cref{cor:prop:cEQ}.
\end{remark}

\begin{remark}{\bf (assumptions in \cref{theorem:cEQ} are quite restrictive)}
	\label{rem:theorem5.5R2}
	Now suppose that $\mathcal{H} =\mathbb{R}^{2}$. Let us assume that $K_{1}$ and $K_{2}$
	are two nonempty closed convex cones in $\mathbb{R}^{2}$ such that $K_{1}$
	or $K_{2}$ is not a linear subspace. Furthermore, assume that $K_{1} \cap K_{2} $ is
	a linear subspace.
	By \cref{theorem:KominusMperpNeq0}\cref{theorem:KominusMperpNeq0:K}, the
	intersections
	$K_{1} \cap K_{2} $ and $K_{1}^{\oplus} \cap K_{2}^{\ominus}$ cannot be both
	equal to $\{(0,0)\}\subseteq \mathbb{R}^2$.
	Because $K_{1}^{\oplus \oplus} \cap K_{2}^{\ominus  \ominus}
	= K_{1} \cap K_{2}$, we assume without loss of generality that  $K_{1} \cap
	K_{2} =\{0\}\times \mathbb{R}$.
	Without loss of generality, we can only have one of the following
	two cases:
	\emph{Case~1}: $K_{1}=\mathbb{R}_+\times\mathbb{R}$ and $K_{2}=\mathbb{R}_-\times\mathbb{R}$;
	\emph{Case~2}:
	$K_{1}= \{0\}\times\mathbb{R}$ and  and $K_{2}=\mathbb{R}_-\times\mathbb{R}$.
	But in either case, we have
	$K_{1}^{\oplus} \cap K_{2}^{\ominus} =\mathbb{R}_{+}\times\{0\}$
	which is \emph{not} a linear subspace.
	
	Therefore, in the Euclidean plane $\mathbb{R}^{2}$, there do not exist two nonlinear cones satisfying
	the assumptions in \cref{theorem:cEQ}!
\end{remark}

\cref{rem:theorem5.5R2} and \cref{l:difference} now prompt
the following natural question with which we conclude this paper:

\begin{question}
	Do there exist nonempty closed convex cones $K_1$ and $K_2$
	that are \emph{nonlinear} yet satisfy
	the assumptions in \cref{theorem:cEQ}?
\end{question}

\vskip 6mm
\noindent{\bf Acknowledgments} 	

\noindent   The authors thank the editor and two anonymous reviewers for their 
helpful and constructive comments.
HHB and XW were partially supported by NSERC Discovery Grants.

\end{document}